\documentclass[11pt]{article}
\usepackage[english]{babel}
\usepackage{geometry} 
\usepackage{color}

\usepackage{latexsym} 
\usepackage{amsmath,amssymb,amsthm}
\usepackage[applemac]{inputenc}

 \newtheorem{thm}{Theorem}[section]

 \newtheorem{prop}[thm]{Proposition}
 \theoremstyle{definition}
 \newtheorem{definition}[thm]{Definition}

 \newtheorem{rem}[thm]{Remark}

 \numberwithin{equation}{section}

\newcommand{\half}{{\textstyle{1\over2}}}
\newcommand{\R}{\mathbb R}%
\newcommand*{\C}{\mathbb C}%

\newcommand{\g}{\mathfrak{g}}
\newcommand{\h}{\mathfrak{h}}
\renewcommand{\j}{j}

\makeatletter
\def\operatorname#1{\mathop{\operator@font #1}\nolimits}%
\makeatother

\DeclareMathOperator{\End}{End}

\DeclareMathOperator{\Ad}{Ad}
\DeclareMathOperator{\id}{Id}

\DeclareMathOperator{\Ker}{Ker}

\newcommand{\Image}{\operatorname{Im}}

\DeclareMathOperator{\Id}{Id}

\newcommand*{\cyclic}{\mathop{\kern0.9ex{{+}\kern-2.2ex\raise-.29ex%
      \hbox{\Large\hbox{$\circlearrowright$}}}}\limits}

\title{ 
Almost complex structures, transverse complex structures, and transverse Dolbeault cohomology. 
}

\author{
Michel Cahen$^1$,  Jean Gutt$^2$ and  Simone Gutt$^1$\\
\scriptsize{michel.cahen@ulb.be, jean.gutt@math.univ-toulouse.fr,   simone.gutt@ulb.be}\\
\footnotesize{$^1$   Universit\'{e} Libre de Bruxelles}\\[-3pt]
\footnotesize Campus Plaine, CP 218, Boulevard du Triomphe\\[-3pt]
\footnotesize BE -- 1050 Bruxelles, Belgium\\[-3pt]
\footnotesize \& Acad\'emie Royale de Belgique.\\
\footnotesize{$^2$  Universit\'e Toulouse III -- Paul Sabatier}\\[-3pt]
\footnotesize 118 route de Narbonne, 31062 Toulouse Cedex 9\\[-3pt]
\footnotesize \& Institut National Universitaire Champollion\\[-3pt]
 \footnotesize  Place de Verdun, 81012 Albi, France\\[-3pt] 
}

\begin{document}

\maketitle
\begin{abstract}  We define  a  transverse Dolbeault cohomology associated to any almost complex structure $j$ on a smooth manifold $M$. This we do by extending the notion of transverse complex structure and by introducing a natural   $\j$-stable involutive limit distribution with such a transverse complex structure. We relate this transverse Dolbeault cohomology  to the  generalized Dolbeault cohomology of $(M,\j)$  introduced by Cirici and Wilson in \cite{bib:CW}, showing that the $(p,0)$ cohomology spaces coincide.
This study of transversality leads us to suggest a notion of minimally non-integrable almost complex structure.
\end{abstract}

\section*{Introduction}
Dolbeault cohomology is defined for a manifold endowed with an integrable almost complex structure.
There have been various ways to extend this cohomology to a manifold endowed with a non integrable almost complex structure.
An almost complex structure $\j$ is a smooth field of endomorphisms of the tangent bundle whose square is minus the identity ($\j^2=-\id$). It induces a splitting of the complexified tangent bundle
$
TM^\C=T_\j^{1,0}\oplus T_\j^{0,1}
$
into $\pm i $ eigenspaces for $\j$,  a corresponding dual splitting of the complexified cotangent bundle and a splitting of complex valued $k$-forms on $M$: $$\Omega^k(M,{\mathbb{C}})=\oplus_{p+q=k}\Omega_j^{p,q}.$$ The exterior differential $d$ has the property that 
$$d\Omega_j^{p,q}\subset \Omega_j^{p-1,q+2}\oplus\Omega_j^{p,q+1}\oplus\Omega_j^{p+1,q}\oplus\Omega_j^{p+2,q-1}$$  and splits accordingly as
$$d=\bar{\mu}\oplus \bar{\partial}\oplus\partial\oplus \mu .$$
An almost complex structure $j$ is integrable if and only if $d=\bar{\partial}\oplus\partial$, i.e. $\bar{\mu}=\mu=0$, which is equivalent to the vanishing of the Nijenhuis tensor (also called torsion) $N^j$ of $j$.
The operator $\bar{\partial}$ defines then the Dolbeault cohomology of the complex manifold $(M,j)$.

An important feature about this cohomology comes from Hodge theory \cite{bib:Hodge} which relates it to harmonic forms. This requires the choice of a Riemannian metric $g$ compatible with $j$.
Many works are devoted to the study of properties of $\bar\partial$-harmonic $(p,q)$-forms or variations of those when $\j$ is not integrable; see for instance \cite{bib:TT,bib:V}, the  recent review by Zhang of  Hodge theory for almost complex manifolds \cite{bib:Zhang}, and papers quoted there.
In Hirzebruch’s 1954 problem list \cite{bib:Hirz}, a question attributed to Kodaira and Spencer concerning these spaces  asks whether the dimension of the space of ${\bar\partial}$-harmonic $(p,q)$-forms depends on the choice of the Riemannian metric, or only on the almost complex structure $j$. In 2020, Holt and Zhang gave in \cite{bib:HZ1} examples showing that the dimension may depend on the choice of the metric. This raises the interest of a cohomology depending only on the almost complex structure $j$.

A solution was given in 2021 by J. Cirici and S. Wilson who gave in \cite{bib:CW} the definition of a generalized Dolbeault cohomology associated to an almost complex structure $j$. It
 is defined as the cohomology of the operator induced by $\bar\partial$ on the cohomology spaces for the operator $\bar\mu$. Some of these cohomology spaces are infinite dimensional in the non integrable case \cite{bib:CPS}.\\

In the present paper, we suggest another Dolbeault cohomology defined only in terms of the almost complex structure $j$, which we call  transverse Dolbeault cohomology. 
 It  is defined as the (usual) Dolbeault cohomology of a natural transverse complex structure induced by the given almost complex structure.\\ 
We consider the involutive $\j$ stable (generalized) distribution defined by  the real part of the limit of the derived flag of distributions associated to the $+i$ eigenspace of $j$.
It is  a natural involutive limit of a sequence of nested real distributions $\mathcal{D}_j^\infty=\cup_k \mathcal{D}_j^{(k)}$ associated to $j$, which contains the image of the Nijenhuis tensor $N^j$. We extend the notion of transverse complex structure to that setting.  The cohomology we define  is the cohomology of the $\bar\partial$ operator restricted to a subspace
of forms which are clearly in the kernel of $\bar\mu$: the forms whose contraction and Lie derivative with respect to a vector field in $\mathcal{D}_j^\infty$ vanish.\\
Equivalently, it is the cohomology of the operator  $\bar\partial$  restricted to the smallest $C^\infty(M,\mathbb{C})$-submodule $\Omega_{\mathcal{D}_j^\infty}(M)$ of smooth complex forms on the manifold, which has the property that it consists of forms  vanishing whenever contracted with a vector field in the image of the Nijenhuis tensor of $j$, which is stable under the differential $d$ and which  splits into $(p,q)$ components relatively to $j$
(i.e. for $\omega$  a $k$-form in the submodule, $\omega\circ j_r$ is in the submodule for any $r\le k$, where $\circ j_r$ indicates the precomposition with $j$ acting on the $r$-th argument of $\omega$).

If one of the derived distribution is involutive and has constant rank and if the space of leaves of 
its real part has a manifold structure making the canonical projection a submersion, then $\j$ induces a complex structure on this quotient manifold  and the cohomology we define coincides with the Dolbeault cohomology of this space of leaves.\\

An almost complex structure is maximally non-integrable \cite{bib:CGGH} when the image of the Nijenhuis tensor at each point $p$ spans the whole tangent space at $p$.
For a maximally non-integrable almost complex structure, all our transverse Dolbeault cohomology spaces vanish.
The first geometrical non-integrable almost complex structure was given by Eells and Salamon in \cite{bib:ES}; it arises on a twistor space when flipping the sign of the vertical part of the standard integrable almost complex structure on this space and it is maximally non-integrable; this property remains true for a similar construction on many twistor spaces \cite{bib:CGR}.
Maximally non integrable almost complex structures   are generic in high dimension :  R. Coelho, G. Placini, and J. Stelzig prove in \cite{bib:CPS}
that in dimension $2n\ge 10$  any almost complex structure on a $2n$-dimensional manifold is homotopic to a maximally non-integrable one. 

A minimality condition for a non-integrable almost complex structure is given by asking the involutivity of the first derived distribution of  the $+i$ eigenspace of $j$. This corresponds to the existence of a maximal transverse complex structure. Examples are given by complex line bundles over complex manifolds with no holomorphic structure with the $\j$ naturally defined when choosing a connection. Another  example in an invariant situation  is given by Thurston's example of a $4$-dimensional compact symplectic  manifold with no K\"ahler structure endowed with an invariant $\j$.\\

In section \ref{section:derived},  we describe the derived flag  of distributions associated to the $+i$ eigenspace of an almost complex structure $j$ and the involutive  real limit distribution  $\mathcal{D}_j^\infty$. \\ 
In section \ref{section:transverse}, we define $\mathcal{D}$-transverse objects when $\mathcal{D}=\Gamma^\infty(D)$ is the space of smooth sections  of a real smooth  involutive distribution $D$ (not necessarily of constant rank) or an involutive  limit  of an increasing sequence of  spaces of sections such as $\mathcal{D}_j^\infty$.
The definitions are chosen so that, in the case where $\mathcal{D}=\Gamma^\infty(D)$ with $D$  of constant dimension and such that the space of leaves for  $D$, denoted by $M/D$, has a manifold structure making the canonical projection  $p: M \rightarrow M/D$  a submersion, $\mathcal{D}$-transverse objects  correspond to objects on the space of leaves $M/D$. 
We define in the general context $\mathcal{D}$-transverse vector fields, $\mathcal{D}$-transverse almost complex   and complex structures, $\mathcal{D}$-transverse forms and $\mathcal{D}$-transverse Dolbeault cohomology.\\
In section \ref{section:transverse Dolbeault cohomology}, we define the transverse Dolbeault cohomology associated to an almost complex structure $\j$ : it is the $\mathcal{D}$-transverse Dolbeault cohomology in the sense of section \ref{section:transverse} when $\mathcal{D}$ is the real limit distribution  $\mathcal{D}_j^\infty$ defined in section \ref{section:derived}. We prove that the $p,0$-cohomology spaces for this transverse cohomology coincide with the $p,0$-cohomology spaces for the Dolbeault cohomology introduced by Cirici and Wilson.\\
In section \ref{section:minimality for non-integrability}, we suggest a minimality condition for a non-integrable almost complex structure and we write this  condition in a homogeneous framework.

\section{Derived distributions associated to an almost complex structure}\label{section:derived}

The {\it{Nijenhuis  tensor, also called torsion, associated to a smooth field $k$ of endomorphisms  of the tangent bundle}} is the tensor of type  
 $(1,2)$ defined by
\begin{equation}
N^k(X,Y):=[kX,kY]-k[kX,Y]-k[X,kY]+k^2[X,Y] \quad \forall X,Y\in \mathfrak{X}(M),
\end{equation}
where $\mathfrak{X}(M)$ is the Lie algebra of ${C}^{\infty}$ vector fields on M.\\
The Newlander-Nirenberg theorem asserts that an almost complex structure $j$ on a manifold $M$  is integrable if and only if its Nijenhuis tensor $N^j$ vanishes identically.

Given an almost complex structure $\j$ on a manifold $M$, the {\it{image distribution}}  on  $M$, denoted $\Image N^j$, has for value at a point $x\in M $ the subspace $(\Image N^j)_x$ of the tangent space $T_xM$ spanned by all values  $N^j_x(X,Y)$. Since $N^j_x(jX,Y)=-jN^j_x(X,Y)$, this distribution is stable by $j$. It is smooth (in the sense that in the neighborhood of any point, there exists a finite  number of smooth vector fields
whose values at any point linearly generate the distribution at that point) but the dimension of the distribution may not be constant.

One has the splitting of the complexified tangent bundle induced by $\j$:
$$
TM^\C=T_\j^{1,0}\oplus T_\j^{0,1}
$$
into $\pm i $ eigenspaces for $\j$.
We shall denote by $\mathcal{T}_\j^\pm$ the sections of the corresponding distributions.
We have the $C^\infty(M,\R)$-linear bijections 
$$
A^\pm : {\mathfrak{X}}(M)\rightarrow \mathcal{T}_\j^\pm : X \mapsto \half (X\mp i \j X).
$$ 
Observe that, for any $X,W\in {\mathfrak{X}}(M)$ one has :
\begin{eqnarray}\label{eq:decomp}
X+iW&=& \half \left((X+\j W)-ij(X+\j W)\right)+\half \left((X-\j W)+i\j (X-\j W)\right)\nonumber\\
&=&A^+(X+\j W)+A^-(X-\j W),
\end{eqnarray}
and
 \begin{eqnarray}
[X-i\epsilon \j X,Y-i\epsilon' \j Y]&=&[X,Y]-\epsilon\epsilon'[\j X,\j Y]-i\left(\epsilon[\j X,Y]+\epsilon'[X,\j Y]\right)\nonumber\\
&=& A^+\left([X,Y]-\epsilon\epsilon'[\j X,\j Y]-\epsilon \j [\j X,Y]-\epsilon' \j [X,\j Y]\right)\label{eq:bracketj}\\
&& + A^-\left([X,Y]-\epsilon\epsilon'[\j X,\j Y]+\epsilon \j [\j X,Y]+\epsilon' \j [X,\j Y]\right).\nonumber
\end{eqnarray}
With $\epsilon=\epsilon'=1$, the above shows that the projection on $\mathcal{T}_\j^-$ of 
$[X-i\j X,Y-i \j Y]$ is given by $A^-(-N^\j(X,Y))$.
Hence  the bracket of two sections in $\mathcal{T}_\j^+$  is always an element of $\mathcal{T}_\j^+$ 
iff the Nijenhuis tensor is zero. A rephrasing of Newlander-Nirenberg's theorem is that $j$ is integrable iff  the distribution $\mathcal{T}_\j^+$ (and hence $\mathcal{T}_\j^-$) is involutive.

When a distribution is not involutive, one extends it to make it involutive in the following way.
\begin{definition}
Given  a smooth real (resp. complex) distribution $D$ whose sections are denoted $\mathcal{D}$, one defines the {\emph{derived flag of the distribution}} as the nested sequence of distributions defined inductively by
$$
\mathcal{D}^{(0)}=\mathcal{D} \qquad \mathcal{D}^{(1)}=\mathcal{D}+[\mathcal{D},\mathcal{D}] \qquad  \mathcal{D}^{(i+1)}=\mathcal{D}^{(i)}+[\mathcal{D}^{(i)},\mathcal{D}^{(i)}].
$$
The limit $\mathcal{D}^{\infty}:=\cup_k \mathcal{D}^{(k)}$ is called \emph{the involutive limit distribution}; it is a $C^\infty (M,\mathbb{R})$ submodule (resp. $C^\infty (M,\mathbb{C})$ submodule) of real (resp. complex) vector fields, and is ``involutive" in the sense that it is a Lie subalgebra of vector fields (i.e. closed under bracket of vector fields).
\end{definition}
The computation made above shows the following.
\begin{prop} The first derived distribution of $\mathcal{T}_\j^+$ is given
by 
$$
\left(\mathcal{T}_\j^+\right)^{(1)}=\mathcal{T}_\j^++[\mathcal{T}_\j^+,\mathcal{T}_\j^+]=\mathcal{T}_\j^+\oplus A^-(\mathcal{I}m\, N^\j)=\mathcal{T}_\j^+ + \left(\mathcal{I}m\, N^\j\right)^\C
$$
where $\mathcal{I}m\, N^\j$ denotes the sections of $\Image N^\j$.
\end{prop}

An almost complex structure $\j$ is said to be {\emph{maximally non-integrable}} if $\Image N^\j=TM$. This happens iff the first derived distribution of $\mathcal{T}_\j^+ $ consist of all complex valued vector fields on $M$.

\begin{prop}\label{prop:kderived}
The $k$-th derived distribution of $\mathcal{T}_j^{1,0}$ can be written as
$$
\left(\mathcal{T}_j^{1,0}\right)^{(k)}=:\mathcal{T}_j^{1,0}\oplus A^-(\mathcal{D}_j^{(k)})=\mathcal{T}_j^{1,0} + \left( \mathcal{D}_j^{(k)}\right)^\C.
$$
where $\mathcal{D}_j^{(1)}=\mathcal{I}m\, N^j$ and the real smooth distributions $\mathcal{D}_j^{(k)}$
are stable under $j$ and defined inductively by 
$$
\mathcal{D}_j^{(k+1)}= \mathcal{D}_j^{(k)}+ \sum_{U\in \mathcal{D}_j^{(k)}}\left(\mathcal{I}m\,{\mathcal{L}}_Uj\right)+[\mathcal{D}_j^{(k)},\mathcal{D}_j^{(k)}],
$$
where $\mathcal{I}m\,{\mathcal{L}}_Uj=\{ \,({\mathcal{L}}_Uj)X\,\vert\, X\in \mathfrak{X}(M)\,\} $ denotes the image of the smooth field of endomorphisms of the tangent bundle given by the Lie derivative of the almost complex structure $j$ in the direction of the vector field $U$.
\end{prop}
 \begin{proof}
Using again $X+iW =A^+(X+jW)+A^-(X-jW)$, we have
\begin{eqnarray*}
[X-i jX,U]&=&[X,U]-i[jX,U]= A^+\left([X,U]-j[jX,U]\right)
 + A^-\left([X,U]+ j[jX,U]\right)\\
 &=& A^+\left([X,U]-j[jX,U]\right) +  A^-\left(({\mathcal{L}}_Uj)(jX)\right).
\end{eqnarray*}
One proves inductively that the distributions $\mathcal{D}_j^{(k)}$ are stable by $j$, observing that $j$ anticommutes with ${\mathcal{L}}_Uj$ since $j^2=-\Id$, and $j[U,U']=[U,jU']-({\mathcal{L}}_Uj)(U')$.
\end{proof}
\begin{prop}\label{prop:t1invol}
The $k$-th derived distribution $\left(\mathcal{T}_j^{1,0}\right)^{(k)}=:\mathcal{T}_j^{1,0}\oplus A^-(\mathcal{D}_j^{(k)})$ is involutive iff $\mathcal{D}_j^{(k)}$ is  involutive  and has the property that $\mathcal{I}m\,{\mathcal{L}}_Uj\subset \mathcal{D}_j^{(k)}$ for each $U\in \mathcal{D}_j^{(k)}$.\\[0.2cm]
The  complex involutive limit distribution 
\begin{equation}
\left(\mathcal{T}_j^{1,0}\right)^{\infty}:=\cup_k \left(\mathcal{T}_j^{1,0}\right)^{(k)}=\mathcal{T}_j^{1,0}\oplus A^-(\cup_k \mathcal{D}_j^{(k)})\label{def:limitcomplex}
\end{equation} is involutive (in the sense that the bracket of two elements in $\left(\mathcal{T}_j^{1,0}\right)^{\infty}$ is again in $\left(\mathcal{T}_j^{1,0}\right)^{\infty}$).\\[2mm]
The  real limit distribution 
\begin{equation}\label{def:limit}
\mathcal{D}_j^\infty:=\cup_k \mathcal{D}_j^{(k)}
\end{equation} is  a $C^\infty (M,\mathbb{R})$ submodule   of real  vector fields; it is   involutive  and has the property that $$\mathcal{I}m\,{\mathcal{L}}_Uj\subset \mathcal{D}_j^{\infty}\quad \forall U\in \mathcal{D}_j^{\infty}.$$
The first derived distribution of $\mathcal{T}_j^{1,0} $, $\left(\mathcal{T}_j^{1,0}\right)^{(1)}=\mathcal{T}_j^{1,0} + \left(\mathcal{I}m\, N^j\right)^\C$,  is involutive iff
\begin{equation}\label{eq:LieN}
[N,jX]-j[N,X]=({\mathcal{L}}_Nj)X  \in \mathcal{I}m\, N^j\qquad \forall X\in {\mathfrak{X}}(M), \forall N \in \mathcal{I}m\, N^j.
\end{equation}
\end{prop}
\begin{proof}
All statements except the last one are  direct consequences of the former proposition.
Equation \ref{eq:LieN} is necessary for $\left(\mathcal{T}_j^{1,0}\right)^{(1)}$ to be involutive; it will be sufficient iff it implies that $\mathcal{I}m\, N^j$ is involutive. This is true, because, for any $N\in \mathcal{I}m\, N^j$ and any $X,Y\in {\mathfrak{X}}(M)$ 
\begin{eqnarray*}  
[N^j(X,Y),N]&=& [[jX,jY],N]-[j[jX,Y],N]-[j[X,jY],N]-[[X,Y],N]\\
&=& [[jX,N],jY]+[jX,[jY,N]] -j[[jX,Y],N]+(\mathcal{L}_Nj)([jX,Y])\\
&&-j[[X,jY],N]+(\mathcal{L}_Nj)([X,jY])-[[X,N],Y]-[X,[Y,N]]\\
&=& (\mathcal{L}_Nj)([jX,Y])+(\mathcal{L}_Nj)([X,jY])\\
&&+[[jX,N],jY]+[jX,[jY,N]]
-j[[jX,N],Y]-j[jX,[Y,N]]\\
&&-j[[X,N],jY]-j[X,[jY,N]]-[[X,N],Y]-[X,[Y,N]]\\
&=&(\mathcal{L}_Nj)([jX,Y])+(\mathcal{L}_Nj)([X,jY])\\
&&+[j[X,N],jY]-[(\mathcal{L}_Nj)(X),jY] +[jX,j[Y,N]]-[jX,(\mathcal{L}_Nj)(Y)]\\
&&-j[j[X,N],Y]+j[(\mathcal{L}_Nj)(X),Y]-j[jX,[Y,N]]\\
&&-j[[X,N],jY]-j[X,j[Y,N]]+j[X,(\mathcal{L}_Nj)(Y)]\\
&&-[[X,N],Y]-[X,[Y,N]]\\
&=&(\mathcal{L}_Nj)([jX,Y])+(\mathcal{L}_Nj)([X,jY])\\
&&-[(\mathcal{L}_Nj)(X),jY]+j[(\mathcal{L}_Nj)(X),Y]\\
&&-[jX,(\mathcal{L}_Nj)(Y)]+j[X,(\mathcal{L}_Nj)(Y)]\\
 && +N^j([X,N],Y) + N^j(X,[Y,N])\\
 &=&(\mathcal{L}_Nj)([jX,Y])+(\mathcal{L}_Nj)([X,jY])\\
 && -(\mathcal{L}_{(\mathcal{L}_Nj)(X)}j)(Y)+(\mathcal{L}_{(\mathcal{L}_Nj)(Y)}j)(X)\\
  && +N^j([X,N],Y) + N^j(X,[Y,N])\\
\end{eqnarray*}
which obviously belongs to $\mathcal{I}m\, N^j$, when $(\mathcal{L}_Nj)(X)\in \mathcal{I}m\, N^j$ for any $N\in \mathcal{I}m\, N^j$.
\end{proof}
\begin{rem}
If the $k$-th derived distribution $\left(\mathcal{T}_j^{1,0}\right)^{(k)}=:\mathcal{T}_j^{1,0}\oplus A^-(\mathcal{D}_j^{(k)})$ is involutive and regular, then  $\mathcal{D}_j^\infty=\mathcal{D}_j^{(k)}$ defines a foliation.\\
In  a homogeneous context, with a $G$-invariant almost complex structure $j$ on a $G$-homogeneous space $M$, each derived distribution is $G$-invariant and regular, so there is always an integer $k$ such that $\left(\mathcal{T}_j^{1,0}\right)^{(k)}$ is involutive. In that context $\mathcal{D}_j^\infty=\mathcal{D}_j^{(k)}$ is a smooth real regular involutive distribution which defines a foliation on $M$.
\end{rem}


\section{$\mathcal{D}$-transverse structures}\label{section:transverse}

Let  $\mathcal{D}=\Gamma^\infty(D)$ be the space of smooth sections  of a real smooth  involutive distribution $D$ (not necessarily of constant rank) on a manifold $M$, or an involutive  limit  of an increasing sequence of such spaces of sections
$\mathcal{D}=\cup_k \mathcal{D}^{(k)}$ with $\mathcal{D}^{(k)}=\Gamma^\infty(D^{(k)})$  and  $ \mathcal{D}^{(k)}+[\mathcal{D}^{(k)},\mathcal{D}^{(k)}]\subset \mathcal{D}^{(k+1)}$, as, for example, the space $\mathcal{D}_j^\infty$ defined in \ref{def:limit}.\\[2mm]
We are going to define $\mathcal{D}$-transverse objects. We start with the particular ideal situation where $\mathcal{D}$ is the space of sections of a distribution $D$ which 
is   regular (i.e. of constant dimension) and such that the space of leaves for  $D$, denoted by $M/D$, has a manifold structure making the canonical projection  $p: M \rightarrow M/D$  a submersion.  In that case, $\mathcal{D}$-transverse objects  translate, at the level of $M$, corresponding objects on the space of leaves $M/D$. We then extend the definitions to our more general setting for $\mathcal{D}$.
\subsection{$\mathcal{D}$-transverse vector fields} 
In the particular ideal situation described above, 
 one can consider the pullback of the tangent bundle $T(M/D)$ and clearly 
$$p^*T\left(M/D\right)\simeq TM/D=:Q.$$ 
The bundle $Q$ is called the normal bundle and one denotes by   $\Pi :TM\rightarrow Q$  the canonical projection.
There is an action  of $\mathcal{D}$  on sections of $Q$: for $F\in \mathcal{D}$ and $u\in \Gamma^\infty(M,Q)$;   
$$L^Q_F u:=\Pi ([F, U]),$$ 
if $U\in {\mathfrak{X}}(M)$ is a lift of $u$ in the sense that $\Pi \circ U=u$. 
A vector field on $M/D$ can be viewed as a section $u$ of $Q$ which is ``constant'' along the leaves  in the sense  that $L^Q_F u=0$  for $F\in \mathcal{D}$. 

In our  general setting we define a transverse vector field as follows.
\begin{definition}
A   $\mathcal{D}$-{{\emph{transverse vector field}}} is an equivalence class $[U]$ of 
a vector field $U\in {\mathfrak{X}}(M)$ which is  $\mathcal{D}$-{\emph{foliated}} in the sense that 
\begin{equation*}
[F,U]\in\mathcal{D} \quad \textrm{ for any}\quad F\in \mathcal{D},
\end{equation*}
the equivalence  of foliated vector fields being defined by $U\sim U'$ iff $U-U'\in \mathcal{D}$.
\end{definition}
\subsection{$\mathcal{D}$-Transverse almost complex  structures}
 In the ideal situation, an almost complex structure $\widehat{j}$ on $M/D$, gives a section $\widetilde{j}$ of $\End(Q)$ which squares to $-\id$ and which  is ``constant" along the leaves in the sense that
  $$L^{\End(Q)}_{{F}}\widetilde{j}=0\, \textrm{  for any }\, F\in \mathcal{D}, \, \textrm{  where  }\,   L^{\End(Q)}_F \widetilde{j}=L^Q_F\circ \widetilde{j}- \widetilde{j}\circ L^Q_F.$$
A lift of 
$\widetilde{j}$ is a section $k\in \End(TM)$ such that $\widetilde{j}(\Pi(U))=\Pi(kU)$;   
remark that 
\begin{itemize}
\item $\Pi(kF)=0$ for all $F\in \mathcal{D}$  if and only if $k(\mathcal{D})\subset \mathcal{D}$,
\item  $\widetilde{j}^2(\Pi (U))=-\Pi (U)$ iff $\Pi(k^2 U)=-\Pi(U)$ iff $k^2U+U\in \mathcal{D}, \, \forall U\in{\mathfrak{X}}(M)$;  
\item $(L^Q_F\circ \widetilde{j}- \widetilde{j}\circ L^Q_F)\Pi(U)=0=\Pi([F,kU]-k[F,U]) \forall U\in{\mathfrak{X}}(M)$.
 \end{itemize} 

In our  general setting we define a transverse almost complex  structure as follows.\begin{definition}\label{def:transvalmostcplx}
A $\mathcal{D}$-{\emph{transverse almost complex structure}} is the equivalence class $[k]$ of a section $k$ of $\End(TM)$  such that $k(\mathcal{D})\subset \mathcal{D}$, $k^2U+U\in \mathcal{D}, \, \forall U\in{\mathfrak{X}}(M)$, and
$$
({\mathcal{L}}_F k)(U)=[F,kU]-k[F,U] \in  \mathcal{D}\quad  \textrm{ for all  } F\in \mathcal{D},\, U \in {\mathfrak{X}}(M),
$$
the equivalence being defined by $k\sim k'$ iff $\mathcal{I}m\,(k-k')\subset \mathcal{D}$.
 \end{definition}
\subsection{$\mathcal{D}$-transverse complex structures}
In the ideal situation, an almost complex structure $\widehat{j}$ on $M/D$ is integrable if and only if its Nijenhuis tensor vanishes. This will be true if and only if the torsion $N^{\widetilde{j}}$
corresponding to the section $\widetilde{j}$ of $\End(Q)$, $N^{\widetilde{j}} (u,v ):= [\widetilde{j}u, \widetilde{j}v ] -\widetilde{j}[\widetilde{j}u,v ]-\widetilde{j} [u, \widetilde{j}v ] -[u,v]$ vanishes  for any $u,v$ sections  of $Q$ corresponding to vector fields on $M/D$, i.e. such that $L^Q_F u=0$ and $L^Q_F v=0$  for $F\in \mathcal{D}$. Another way to formulate this condition is that  a lift $k\in \End(TM)$ of
$\widetilde{j}$ must satisfy
$$\Pi\left([kU, kV ] -k[kU,V ]-k [U, kV ] +k^2[U,V]\right) =\Pi\left(N^k(U,V)\right)=0$$ 
for any $U,V$ foliated vector fields. Since the torsion $N^k$ of $k$  is a tensor and since the value at a point of foliated vector fields generate the whole tangent space to $M$ at that point in this ideal situation, it is equivalent to ask that $N^k$ takes its values in $\mathcal{D}$.

 In our  general setting we define a transverse  complex  structure as follows.
 \begin{definition}\label{def:transvcplx}
A $\mathcal{D}$-{\emph{transverse  complex structure}} is a $\mathcal{D}$-transverse almost complex structure  $[k]$ 
(i.e. $k$ is a section of $\End(TM)$, such that
 $k(\mathcal{D})\subset \mathcal{D}$, $k^2U+U\in \mathcal{D}$, and
$ [F,kU]-k[F,U] \in  \mathcal{D}\,\textrm{ for all  } F\in \mathcal{D}$ and all $ U \in {\mathfrak{X}}(M)$) 
which has the property that 
$$
N^k(U,V)=[kU, kV ] -k[kU,V ]-k [U, kV ] +k^2[U,V] \in \mathcal{D}, \textrm{ for all  } U,V \in {\mathfrak{X}}(M).
$$
 \end{definition}  
\subsection{$\mathcal{D}$-transverse forms}
 \begin{definition}
A $\mathcal{D}$-{\emph{transverse -or basic- real or complex  $p$-form }} is  a $p$-form $\omega$ on $M$ such that
$$ \iota(F)\omega=0 \qquad \textrm{and}\qquad {\mathcal{L}}_F\omega=0 \qquad \forall F \in \mathcal{D}.
$$
If the section  $k$ of $\End(TM)$ defines a $\mathcal{D}$-transverse almost complex structure,
a complex $\mathcal{D}$-transverse $1$-form $\omega $ is {\emph{of type $(1,0)$ --resp. $(0,1)$--}} iff 
$$
\omega(U+ikU)=0, \qquad  \textrm{-- resp.}\,\,  \omega(U-ikU)=0 \qquad \forall U\in {\mathfrak{X}}(M).
$$
  \end{definition}
  \begin{prop} Given a $\mathcal{D}$-transverse almost complex structure $[k]$,
there is a  splitting of complex $\mathcal{D}$-transverse p-forms as a direct sum
$$\Omega^\ell_\mathcal{D}(M)=\oplus_{p+q=\ell}\Omega_{\mathcal{D},[k]}^{p,q}.$$
\end{prop}
\begin{proof}  
Observe that
any complex $\mathcal{D}$-transverse $1$-form $\omega$  splits as $$\omega=\half(\omega-i\,\omega\circ k)+\half(\omega+i\,\omega\circ k),$$
and this decomposition does not depend on the element $k\in [k]$.
The form $\omega\circ k$ is $\mathcal{D}$-transverse since $\omega\circ k(F)=0$ for $F\in \mathcal{D}$ because $kF\in\mathcal{D}$ and 
$${\mathcal{L}}_F(\omega\circ k)={\mathcal{L}}_F(\omega)\circ k +\omega\circ {\mathcal{L}}_F k =0$$ since ${\mathcal{L}}_F k (U)\in \mathcal{D}$ for any $U \in {\mathfrak{X}}(M)$.
Furthermore
$$(\omega\mp i\omega\circ k)(U\pm ikU)=\omega(U+k^2U)+i \omega(\mp kU \pm kU)=0 \quad \forall U \in {\mathfrak{X}}(M).$$
The same applies for complex $\mathcal{D}$-transverse $\ell$-forms. 
\end{proof}



  \subsection{$\mathcal{D}$-transverse Dolbeault cohomology}
Assume again that the section  $k$ of $\End(TM)$ defines a $\mathcal{D}$-transverse almost complex structure.
The differential of a $\mathcal{D}$-transverse form is again  $\mathcal{D}$-transverse, and one has
 $$d\Omega_{\mathcal{D},[k]}^{p,q}\subset \Omega_{\mathcal{D},[k]}^{p+2,q-1}\oplus\Omega_{\mathcal{D},[k]}^{p+1,q}\oplus\Omega_{\mathcal{D},[k]}^{p,q+1}\oplus \Omega_{\mathcal{D},[k]}^{p-1,q+2}.$$ 
Indeed, for $\omega \in \Omega_{\mathcal{D},[k]}^{p,q}$ and for complex vector fields  $Z_i$,  
$$
d\omega(Z_0,..,Z_p)=\sum_r (-1)^r Z_r(\omega(Z_0,..\widehat{Z_r}..,Z_p))+\sum_{r<s} (-1)^{r+s} \omega([Z_r,Z_s],Z_0,..\widehat{Z_r}..\widehat{Z_s}..,Z_p)
$$
 so that, for elements $Z_r={A_k^{\pm}Y_r}=Y_r\mp ikY_r$, it vanishes if there are not at least $p-1$ elements of the form $Z_r={A_k^{+}Y_r}$ and $q-1$ elements of the form  $Z_r={A_k^{-}Y_r}$.
 
 Assume now that the section  $k$ of $\End(TM)$ defines a $\mathcal{D}$-transverse  complex structure.
 Then  the projections of $d\omega$ on $\Omega_{\mathcal{D},[k]}^{p+2,q-1}$ and $\Omega_{\mathcal{D},[k]}^{p-1,q+2}$ vanish; indeed
 \begin{eqnarray*}
 &&d\omega(A_k^+X_1,..., A_k^+X_{p-1},A_k^-Y_1,..A_k^-Y_{q+2})\\
 &=& \sum\pm\omega([A_k^-Y_r,A_k^-Y_s],A_k^+X_1,..., A_k^+X_{p-1},A_k^-Y_1,..\widehat{A_k^-Y_r}..\widehat{A_k^-Y_s}..,A_k^-Y_{q+2})=0
  \end{eqnarray*}
 because, 
  \begin{eqnarray*}[Y_r + ikY_r, Y_s+ikY_s]&=&[Y_r, Y_s]-[kY_r, kY_s]+i[Y_r, kY_s]+i[kY_r, Y_s]
 \\
 &=&A_k^+([Y_r, Y_s]-[kY_r, kY_s]+k[Y_r, kY_s]+k[kY_r, Y_s])+\\
 && \quad A_k^-([Y_r, Y_s]-[kY_r, kY_s]-k[Y_r, kY_s]-k[kY_r, Y_s])\\
 && \quad +i(k^2+\Id)([Y_r, kY_s]+[kY_r, Y_s])\\
 &=&  A_k^+(-N^k(Y_r,Y_s))+A_k^-(...) + F \,\textrm{ with }\, F\in \mathcal{D}^\mathbb{C}
 \end{eqnarray*}
and $N^k(Y_r,Y_s)$ is in $\mathcal{D}$ when $k$ defines a complex transverse structure.
 
 \bigskip
 Hence 
 $$d\Omega_{\mathcal{D},[k]}^{p,q}\subset \Omega_{\mathcal{D},[k]}^{p+1,q}\oplus\Omega_{\mathcal{D},[k]}^{p,q+1},$$ 
and we  denote  by $\partial_{\mathcal{D},[k]}$ and $\bar{\partial}_{\mathcal{D},[k]}$ the corresponding projections.
\begin{definition}
The  $\mathcal{D}$-transverse Dolbeault cohomology induced by $[k]$ is 
$$H_{\mathcal{D},[k],\bar{\partial}}^{p,q}(M)=\Ker \bar{\partial}_{\mathcal{D},[k]}\vert_{\Omega_{\mathcal{D},[k]}^{p,q}}/\Image \bar{\partial}_{\mathcal{D},[k]}\vert_{\Omega_{\mathcal{D},[k]}^{p,q-1}}.$$

 \end{definition} 
 
\begin{rem}
In the ideal situation where $D$ is regular and there is a manifold structure on the space of leaves $M/D$ such that $p: M \rightarrow M/D$ is a submersion, if $[k]$ corresponds to a complex structure $\widehat{j}$ on $M/D$, then the transverse Dolbeault cohomology is the usual Dolbeault cohomology on the space of leaves:
 $$H_{\mathcal{D},[k],\bar{\partial}}^{p,q}(M)=H_{\widehat{j},\bar{\partial}}^{p,q}(M/D).$$ 
\end{rem}
 
\section {Transverse Dolbeault cohomology. }\label{section:transverse Dolbeault cohomology}

\subsection{Transverse complex structure induced by an almost complex structure}
Let $j$ be an almost complex structure on the manifold $M$ and let $\mathcal{D}$ be a real generalized involutive distribution,  stable under $j$; by  this we mean that $\mathcal{D}$ can be the space of smooth sections  of a  smooth  involutive distribution $D$ (not necessarily of constant rank), stable under $j$,  or -as before- can be the involutive  limit  of a sequence of nested spaces of sections.  

Following  definitions  \ref{def:transvalmostcplx} and \ref{def:transvcplx}, this $j$ yields a $\mathcal{D}$-transverse almost complex structure iff
 $$
[F,jU]-j[F,U] \in \mathcal{D} \quad \forall F\in \mathcal{D}, \quad \forall U\in {\mathfrak{X}}(M),
$$
and a complex $\mathcal{D}$-transverse structure iff, furthermore, $\mathcal{D}$ contains the image of $N^j$.
 \begin{prop}\label{prop:transvj}
$j$ defines a complex $\mathcal{D}$-transverse structure iff
$$
\mathcal{T}_j^{1,0}+\mathcal{D}=\mathcal{T}_j^{1,0}\oplus A^-(\mathcal{D})=\mathcal{T}_j^{1,0}+\mathcal{D}^\C  \quad {\textrm{is involutive }}.
$$
  Then  $\mathcal{D} \supset \mathcal{D}_j^\infty=\cup_k \mathcal{D}_j^{(k)}\supset\Image N^j$.  
 \end{prop}
 \begin{proof}
This results directly from the computations made in the proof of proposition \ref{prop:kderived}.
\end{proof}
 In particular $j$ defines  a complex structure transverse to $\mathcal{D}_j^\infty$ and a corresponding $\mathcal{D}_j^\infty$-transverse Dolbeault cohomology. The splitting of $\mathcal{D}_j^\infty$-transverse forms corresponds to the usual splitting of forms on $M$ relatively to $j$,
 $\Omega^k(M,{\mathbb{C}})=\oplus_{p+q=k}\Omega_j^{p,q}$, restricted to $\mathcal{D}_j^\infty$-transverse forms.
 
 \subsection {Transverse Dolbeault cohomology and $(p,0)$-spaces. }
 \begin{definition}
 The {\emph{ $j$-transverse Dolbeault cohomology}} is the $\mathcal{D}_j^\infty$-transverse Dolbeault cohomology induced by $j$
  $$H_{j\, trans}^{p,q}(M):=H_{\mathcal{D}_j^\infty,[j],\bar{\partial}}^{p,q}(M)=\Ker \bar{\partial}_{\mathcal{D}_j^\infty,[j]}\vert_{\Omega_{\mathcal{D}_j^\infty,[j]}^{p,q}}/\Image \bar{\partial}_{\mathcal{D}_j^\infty,[k]}\vert_{\Omega_{\mathcal{D}_j^\infty,[j]}^{p,q-1}},$$
and  $\Omega_{\mathcal{D}_j^\infty,[j]}^{p,q}=\{ \, \omega \in \Omega_j^{p,q}\,\vert\, \iota(X)\omega=0,\, \mathcal{L}_X\omega=0,\, \forall X \in D_j^\infty\,\}.$
\end{definition}

\begin{rem}
If $\mathcal{D}_j^\infty$ is regular, it defines a foliation with transverse complex structure induced by $j$.
If its space of leaves has a manifold structure making the canonical projection a submersion, then the $j$ transverse Dolbeault cohomology is the Dolbeault cohomology of this space of leaves.
Recall that for a $G$-invariant almost complex structure $j$ on a $G$-homogeneous space $M$, each derived distribution is $G$-invariant and regular, so  $\mathcal{D}_j^\infty$  always defines a foliation. 
\end{rem}

 \begin{prop} Let $\mathcal{D}$ be any regular involutive $j$-stable distribution  such that the space of leaves $M/D$ has a manifold structure with 
 $p:M\rightarrow M/D$ a submersion, and such that $j$ induces a complex structure on $M/D$.
 The $\mathcal{D}$-transverse Dolbeault cohomology induced by $j$ coincides with the Dolbeault cohomology of the space of leaves, and maps into the $j$-transverse Dolbeault cohomology.
 \end{prop}

 \begin{prop}
 The space $\Omega_{\mathcal{D}_j^\infty}(M):=\oplus_{p,q}\Omega_{\mathcal{D}_j^\infty,[j]}^{p,q}$ is the smallest $C^\infty(M,\mathbb{C})$-submodule  of the space $\Omega(M,\mathbb{C})$ of smooth complex forms on the manifold such that
 \begin{itemize}
 \item each form $\omega$ in it  vanishes whenever contracted with a vector field in the image of the Nijenhuis tensor of $j$, 
 $$
\iota(X)\omega=0,\,  \forall X \in \mathcal{I}m N^j;
 $$
 \item it is stable under the differential $d$; 
\item it splits into $(p,q)$ components relatively to $j$,
in the sense that for $\omega$  a $k$-form in the submodule, $\omega\circ j_r$ is in the submodule for any $r\le k$, where $\circ j_r$ indicates the precomposition with $j$ acting on the $r$-th argument of $\omega$.
\end{itemize}
\end{prop}

\begin{proof}
 All conditions are clearly satisfied by 
 $$
\Omega_{\mathcal{D}_j^\infty}(M)=\oplus_{p,q} \left\{ \, \omega \in \Omega_j^{p,q}\,\vert\, \iota(X)\omega=0,\, \mathcal{L}_X\omega=0,\, \forall X \in D_j^\infty\,\right\}.
 $$
 Reciprocally, one proceeds by induction, showing that a space of forms satisfying all three conditions is included, for all $k$, in
 $$
 \left\{ \, \omega \in \Omega_j^{p,q}\,\vert\, \iota(X)\omega=0,\, \mathcal{L}_X\omega=0,\, \forall X \in D_j^{(k)} \, \right\}
 $$
For a given $k$, let us consider a subspace of forms such that  $\iota(X)\omega=0,\,  \forall X \in \mathcal{D}_j^{(k)}$, remembering that $D_j^{(1)}=\mathcal{I}m N^j$:\\
 The second condition implies that $\iota(X)d\omega=\mathcal{L}_X\omega=0,\,\forall X \in  \mathcal{D}_j^{(k)}$. Thus, one also has that $\iota(Y)\omega =0$ for any $Y\in \mathcal{D}_j^{(k)}+[\mathcal{D}_j^{(k)},\mathcal{D}_j^{(k)}]$ (using the fact that $\iota([X,X'])\omega= (\iota(X)\circ\mathcal{L}_{X'}-\mathcal{L}_{X'}\circ\iota(X))\omega$).\\
 The third condition implies then that $\omega\circ \mathcal{L}_Xj_r=0$ which in turns implies that 
 $\iota(Y)\omega=0$   for all $Y \in  \mathcal{I}m \mathcal{L}_Xj$ when $X\in \mathcal{D}_j^{(k)}$.\\
 This shows that $\iota(Y)\omega=0$  for all $Y \in \mathcal{D}_j^{(k+1)}$ and one proceeds  inductively.
\end{proof}

 \begin{prop} \label{hp0trans}  The $j$-transverse Dolbeault cohomology is given in degree $p,0$  by
$$H_{j\, trans}^{p,0}(M):=\left\{ \omega \in \Omega^p(M,\mathbb{C})\,\vert\, \iota(Z)\omega=0,\, \mathcal{L}_Z\omega=0, \,\forall Z\in \left(\mathcal{T}_j^{0,1}\right)^{\infty}=\mathcal{T}_j^{0,1}\oplus A^-(D_j^\infty)\right\}.$$
\end{prop}
\begin{proof}
\begin{eqnarray*}
H_{\mathcal{D}_j^\infty,[j],\bar{\partial}}^{p,0}(M)&=&\left\{  \omega \in \Omega_{\mathcal{D}_j^\infty,[j]}^{p,0}\,\vert\,\bar{\partial}_{\mathcal{D}_j^\infty,[j]}\omega=0 \right\}\\  
&=& \left\{  \omega \in \Omega^p(M,\mathbb{C})\,\vert\, \iota(F)\omega=0,\, \mathcal{L}_F\omega=0 \,\forall F\in \mathcal{D}_j^\infty,\right.\\
&&\hspace{1cm}
\left. \iota(U+ijU)\omega=0 \, \forall U\in {\mathfrak{X}}(M), \qquad  \bar{\partial}_{\mathcal{D}_j^\infty,[j]}\omega=0 \right\}\\
&=&\left\{  \omega \in \Omega^p(M,\mathbb{C})\,\vert\, \iota(Z)\omega=0,\,\forall Z \in \mathcal{T}_j^{0,1}+\mathcal{D}_j^{\infty},\,\, \mathcal{L}_F\omega=0, \,\forall F\in \mathcal{D}_j^{\infty}, \right.\\
&&\hspace{2.4cm}
 \left. d\omega(U+ijU,V_1-ijV_1,\ldots,V_{p-1}-ijV_{p-1})=0 \right\}\\  
&=&\left\{  \omega \in \Omega^p(M,\mathbb{C})\,\vert\,  \iota(Z)\omega=0,\,\forall Z \in \mathcal{T}_j^{0,1}+\mathcal{D}_j^{\infty},\,\, \mathcal{L}_F\omega=0, \,\forall F\in \mathcal{D}_j^{\infty}, \right.\\
&&\hspace{-2cm}
\left. {\tiny{ (A^-U)(\omega(A^+V_1,\ldots,A^+V_{p-1}))-\sum_k(\omega(A^+V_1,\ldots,[A^-U,A^+V_k],\ldots, A^+V_{p-1}))=0 }}\right\} \\  
&=&\left\{  \omega \in \Omega^p(M,\mathbb{C})\,\vert\, \iota(Z)\omega=0,\, \mathcal{L}_Z\omega=0, \,\forall Z\in \left(\mathcal{T}_j^{0,1}\right)^{\infty}=\mathcal{T}_j^{0,1}+\mathcal{D}_j^{\infty}\right\} 
\end{eqnarray*}
\end{proof}

\subsection{Comparison with the generalized Dolbeault cohomology of an almost complex structure}

J.Cirici and S Wilson introduced in \cite{bib:CW} a generalized Dolbeault cohomology associated to an almost complex structure $j$ on a manifold $M$ in the following way.

One decomposes as before  the complexified tangent bundle $TM^\C=T_j^{1,0}\oplus T_j^{0,1}$, into $\pm i $ eigenspaces for $j$, and  the dual decomposition of the complexified cotangent bundle $T^*M^\C=(T^*_j)^{1,0}\oplus (T^*_j)^{0,1}$  leads to the usual decomposition of the space of complex valued $k$-forms on $M$ into
$$\Omega^k(M,{\mathbb{C}})=\oplus_{p+q=k}\Omega_j^{p,q}.$$
Then 
$$d\Omega_j^{p,q}\subset \Omega_j^{p-1,q+2}\oplus\Omega_j^{p,q+1}\oplus\Omega_j^{p+1,q}\oplus\Omega_j^{p+2,q-1}$$
and the differential splits accordingly as
$$d=\bar{\mu}\oplus \bar{\partial}\oplus\partial\oplus \mu .$$
The fact that $d^2=0$ is equivalent to
\begin{eqnarray}
\bar\mu^2 &=& 0,\label{eq:mu2}\\
\bar\mu\circ \bar\partial +\bar\partial\circ \bar\mu &=&0, \label{eq:mud}\\
\bar\mu\circ {\partial}+{\partial}\circ \bar\mu + \bar\partial^2 &=& 0,
 \label{eq:delta2}\\
 \mu\circ \bar{\mu}+\bar{\mu}\circ \mu +\partial\circ \bar{\partial}+\bar{\partial}\circ\partial &=& 0 \nonumber\\
 \mu\circ \bar{\partial}+\bar{\partial}\circ \mu + \partial^2&=&0, \nonumber\\
 \mu\circ \partial +\partial\circ \mu &=& 0,\nonumber\\  
 \mu^2 &=& 0.\nonumber
\end{eqnarray}  
Equation \eqref{eq:mu2} shows that one can define the $\bar\mu$ cohomology spaces:
\begin{equation}
H_{\bar\mu}^{(p,q)}=\Ker\bar\mu\vert_{\Omega_j^{p,q}}/\Image \bar\mu\vert_{\Omega_j^{p+1,q-2}}.
\end{equation}
Equation \eqref{eq:mud} shows that $\bar\partial$ induces a map $\widetilde{\bar\partial}$ on those $\bar\mu$ cohomology space
\begin{equation}
 \widetilde{\bar\partial} :H_{\bar\mu}^{(p,q)}\rightarrow H_{\bar\mu}^{(p,q+1)}: \omega +\Image \bar\mu\mapsto \bar\partial\omega +\Image \bar\mu,
 \end{equation} 
and equation \eqref{eq:delta2} shows that $\left(\widetilde{\bar\partial}\right)^2=0$, so one can look at the corresponding cohomology spaces 
 \begin{equation} 
H_{Dol}^{(p,q)}(M) =\Ker\widetilde{\bar\partial}\vert_{H_{\bar\mu}^{(p,q)}}/\Image \widetilde{\bar\partial}\vert_{H_{\bar\mu}^{(p,q-1)}}
\end{equation}
Those are the spaces of the generalized Dolbeault cohomology.

\begin{prop} The generalized Dolbeault and the transverse Dolbeault cohomology spaces coincide in degrees $(p,0)$ :
$$H_{Dol}^{(p,0)}(M)= H_{j\, trans}^{p,0}(M)$$
\end{prop}
\begin{proof}
We know that $H_{j\, trans}^{p,0}(M)= H_{\mathcal{D}_j^\infty,[j],\bar{\partial}}^{p,0}(M)$ and, by proposition \ref{hp0trans} we have 
 $$H_{j\, trans}^{p,0}(M):=\left\{ \omega \in \Omega^p(M,\mathbb{C})\,\vert\, \iota(Z)\omega=0,\, \mathcal{L}_Z\omega=0, \,\forall Z\in \left(\mathcal{T}_j^{0,1}\right)^{\infty}=\mathcal{T}_j^{0,1}\oplus A^-(D_j^\infty)\right\}.$$
 On the other hand
$$H_{Dol}^{(p,0)}(M)=\left\{\omega\in \Omega_j^{p,0}\,\vert\, \bar\mu\omega=0, \bar{\partial}\omega=0\right\}.$$
Now, for $\omega\in \Omega_j^{p,0}$, one has 
\begin{eqnarray*}
\bar\mu\omega=0 &\iff& d\omega (A^-(Y),A^-(Z),A^+(X_1),...,A^+(X_{p-1}))=0\\
& \iff &
\omega([Y+ijY,Z+ijZ],X_1-ijX_1,....,X_{p-1}-ijX_{p-1})=0 \\
&\iff & \iota(N^j(Y,Z))\omega=0.
\end{eqnarray*}
 So
 $$H_{\bar\mu}^{(p,0)}=\left\{\omega \in \Omega^p(M,{\mathbb{C}})\, \vert\, \iota(W)\omega=0 \, \, \forall W\in \left(\mathcal{T}_j^{0,1}\right)^{(1)}\right\}.$$ 
 We have, for $\omega\in H_{\bar\mu}^{(p,0)}\subset \Omega_j^{p,0}$
 \begin{eqnarray*}
 \bar\partial\omega=0 &\iff& d\omega (Y+ijY,A^+(X_1),...,A^+(X_{p}))=0\\
 & \iff  &
(Y+ijY)\omega(A^+(X_1),...,A^+(X_{p}))\\
&& \qquad\qquad -\sum_i \omega(A^+(X_1),.,[Y+ijY,A^+(X_i)],...,A^+(X_p))=0 \\
&\iff &{\mathcal{L}}_Z\omega=0 \,\,\,  \forall Z\in \mathcal{T}_j^{0,1}.
\end{eqnarray*}
Using the fact that $[{\mathcal{L}}_Z,{\mathcal{L}}_{Z'}]={\mathcal{L}}_{[Z,Z']}$ and $[\iota (W),{\mathcal{L}}_Z]= \iota([W,Z])$,  we get 
\begin{eqnarray*}H_{Dol}^{(p,0)}(M)&=&\left\{\omega \in \Omega^p(M,{\mathbb{C}})\,\vert\, \iota(W)\omega=0 \,\textrm{ and }\,  {\mathcal{L}}_Z\omega=0\,\,\forall  Z\in \mathcal{T}_j^{0,1}, \forall W\in \left(\mathcal{T}_j^{0,1}\right)^{(1)} \right\}\\   
&=&\left\{\omega \in \Omega^p(M,{\mathbb{C}})\,\vert\, \iota(Z)\omega=0 \,\textrm{ and }\, {\mathcal{L}}_Z\omega=0\,\,\forall  Z\in \left(\mathcal{T}_j^{0,1}\right)^{\infty}=\cup_k\left(\mathcal{T}_j^{0,1}\right)^{(k)}\right\}\\
&=&H_{\mathcal{D}_j^\infty,\bar{\partial}}^{p,0}(M).
\end{eqnarray*}
\end{proof}
\begin{rem}
Since any element in $\Omega_{\mathcal{D}_j^\infty}(M)=\oplus_{p,q}\Omega_{\mathcal{D}_j^\infty,[j]}^{p,q}(M)$ is in the kernel of $\bar\mu$, there is always a map from $\Omega_{\mathcal{D}_j^\infty}(M)$ to $H_{\bar\mu}^{(p,q)}$ mapping an element $\omega$ to $[\omega]$ and this induces a map in cohomology:
$$H_{j\, trans}^{p,q}(M)\rightarrow H_{Dol}^{(p,q)}(M) $$
mapping the class in $H_{j\, trans}^{p,q}(M)$ of a $\bar\partial$-closed $(p,q)$-form $\omega$ in $\Omega_{\mathcal{D}_j^\infty}(M)$ to the class in $H_{Dol}^{(p,q)}(M)$ of the $\widetilde{\bar\partial}$-closed element $[\omega]$ in $H_{\bar\mu}^{(p,q)}$.
\end{rem}
\section{A notion of minimal non-integrability}
\label{section:minimality for non-integrability}

\begin{definition}
We say that a non-integrable almost complex structure $j$ on a manifold $M$ is minimally non-integrable if the first derived distribution of $\mathcal{T}_j^{1,0} $, $\left(\mathcal{T}_j^{1,0}\right)^{(1)}=\mathcal{T}_j^{1,0} + \left(\mathcal{I}m\, N^j\right)^\C$,  is involutive.  By proposition \ref{prop:t1invol}, this will be true iff
$$
[N,jX]-j[N,X]=({\mathcal{L}}_Nj)X  \in \mathcal{I}m\, N^j\qquad \forall X\in {\mathfrak{X}}(M), \forall N \in \mathcal{I}m\, N^j.
$$
We could ask furthermore that $\dim\mathcal{I}m\, N^j=2$ everywhere. Then we have a foliation, with two-dimensional leaves carrying a complex structure, and with a transverse complex structure.
\end{definition}
As we have seen in prop \ref{prop:transvj}, minimal non-integrability for $\j$ means that  it has the largest possible  transverse complex structure. 


\subsection{Minimally non-integrable   invariant almost complex structure on a Lie group.}

Let $j$ be a left invariant almost complex structure on a Lie group $G$.
Denoting by $\g$ the Lie algebra of $G$ and by $J$ the endomorphism of $\g$ given by the value of $j$ at the neutral element $e\in G$, the Nijenhuis tensor $N^j$ is left invariant and its value
at $e$ is given by 
$$
N^J(X,Y):=[JX,JY]-J[JX,Y]-J[X,JY]-[X,Y]  \qquad X,Y \in \g.
$$
$\Image N^j$ is a smooth left invariant regular distribution whose value at $e$ is the $J$-invariant subspace $\Image N^J$ of $\g$.
\begin{prop}
Consider a left invariant almost complex structure $j$ on a Lie group $G$ such that $\Image N^J$
satisfies 
\begin{equation}\label{LNNJ}
[N,JX]-J[N,X]\in \Image N^J, \quad \forall X\in \g,\, N \in \Image N^J.
\end{equation}
Then $\Image N^j$ defines a foliation, the leaves carry an induced almost complex structure, and $j$ induces a transverse complex structure.\\
If the subgroup $H$ corresponding to the subalgebra $\Image N^J$ is closed, one has a principal fiber bundle $p: G\rightarrow G/H$, whose base manifold (which is the leaf space) is  complex, the fibers (leaves) are almost complex and the projection is pseudo-holomorphic.
The fibers are complex if, furthermore, $N^J(N,N')=0$ for all $N,N'\in \Image N^J$; this is always true if $\dim \Image N^J=2$.
\end{prop}
Remark that  condition \eqref{LNNJ} implies that $\Image N^J$ is a subalgebra of $\g$ and is automatically satisfied if $\Image N^J$ is an ideal in $\g$.
\subsection{Minimally non-integrable  invariant almost complex structure on a homogeneous space.}

Let $G\times M\rightarrow M : (g,x)\mapsto g\cdot x=:\rho(g)x$ denote the action of a Lie group $G$ on a manifold $M$. Assume this action is transitive. Choosing a base point $x_0$, its  stabilizer will be denoted by $H$, so $H=\{ g\in G \, \vert \, g\cdot x_0=x_0\}$ is a closed subgroup of $G$ and the manifold $M$ is diffeomorphic to $G/H$. We denote by $\pi:G\rightarrow M : g\mapsto g\cdot x_0$ the canonical projection. Thus $T_{x_0}M$ identifies with $\g/ \h$. For any element $h\in H$, the action $\rho(h)_{*x_0}$ on $T_{x_0}M$ identifies with the action induced by $\Ad(h)$ and denoted $\widetilde{\Ad(h)}$ on the quotient $\g/ \h$.
We denote by $A^*$ the fundamental vector field on $M$ defined by an element $A\in \g$ (i.e. $A^*_x=\frac{d}{dt} \exp -tA\cdot x\vert_{t=0}$); clearly $A^*_{x_0}=-\pi_{*e}A$.

Assume  that there exists an almost complex structure $j$ on $M$ which is $G$-invariant, i.e.
$\rho(g)_{*x}\circ j_x= j_{g\cdot x}\circ \rho(g)_{*x}$ for all $g\in G$. 
Invariance  implies
\begin{equation}\label{eq:invj}
L_{A^*}j=0 \qquad \textrm{i.e.}\quad [A^*, jX]=j[A^*,X] \qquad \forall X\in {\mathfrak{X}}(M),\qquad \forall A\in \g.
\end{equation}
The Nijenhuis tensor $N^j$ is  invariant under the action of $G$. Its value at the base point is obtained as follows
\begin{prop}(\cite{KN}, thm 6.4, page 217)\label{prop:NJhomog}
Let $M$ be  a $G$-homogeneous  manifold  endowed with a $G$-invariant  almost complex structure $j$. We choose a base point $x_0 \in M$ and a linear map  $J :\g\rightarrow \g$  such that $j_{x_0}A^*=(J A)^*_{x_0}$ for all $A\in \g$. Then
\begin{equation}\label{eq:NJinv}
N^j_{x_0}(A^*,B^*)=\left(-N^{J}(A,B)\right)^*_{x_0}
\end{equation}
 where $N^{J}$ is defined in terms of the Lie bracket in $\g$ by 
$$
N^J(A,B):=[JA,JB] -J[JA,B]-J[A,JB]-[A,B] \quad \forall A,B \in\g.
$$
\end{prop}
\begin{proof}
Using the invariance of $j$, we have, for all $A,B\in \g$,
 \begin{equation*}
 N^j(A^*,B^*)=[jA^*,jB^*]+j [B^*, jA^*]-j[A^*,jB^*]-[A^*,B^*]=[jA^*,jB^*]+[A^*,B^*]  
  \end{equation*}
We introduce an auxiliary torsion-free linear connexion $\nabla$ on the manifold $M$.  At the base point $x_0$, one has
\begin{eqnarray*}
[jA^*,jB^*]_{x_0}&=&(\nabla_{jA^*}jB^*-\nabla_{jB^*}jA^*)_{x_0}=(\nabla_{(JA)^*}jB^*-\nabla_{(JB)^*}jA^*)_{x_0} \\
&=& \left([(JA)^*,jB^*] +\nabla_{jB^*}(JA)^*-[(JB)^*,jA^*]-\nabla_{jA^*}(JB)^*\right)_{x_0}\\
&=&\left(j[(JA)^*,B^*]+\nabla_{(JB)^*}(JA)^*-j[(JB)^*,A^*]-\nabla_{(JA)^*}(JB)^*\right)_{x_0}\\
&=&\left(j[JA,B]^*+[(JB)^*,(JA)^*]-j[JB,A]^*\right)_{x_0} \\
&=&\left(J[JA,B]+[JB,JA]-J[JB,A]\right)^*_{x_0}.
\end{eqnarray*}
 Hence, $N^j(A^*,B^*)_{x_0}=\left(-[JA,JB]+J[JA,B]+J[A,JB]+[A,B]^*\right)^*_{x_0}=-N^J(A,B)^*_{x_0}$.
 \end{proof}
 \begin{prop}
 Let $M$ be  a $G$-homogeneous  manifold  endowed with a $G$-invariant  almost complex structure $j$. We choose a base point $x_0 \in M$ and a linear map  $J :\g\rightarrow \g$  such that $j_{x_0}A^*=(J A)^*_{x_0}$ for all $A\in \g$. The distribution $\Image N^j$ is involutive and gives a foliation with transverse complex structure if and only if 
\begin{equation}
 [JA,N^J(B,C)]-J[A, N^J(B,C)] \in \Image N^J+\h \qquad \forall A,B,C \in \g,
 \end{equation}
with $N^J$ defined as above  $N^J(A,B)=[JA,JB] -J[JA,B]-J[A,JB]-[A,B]$. 
 \end{prop}
 \begin{proof}
  Since $j$ in $G$-invariant, the tensor $N^j$ is $G$-invariant, hence, for all $A\in\g$,
 \begin{equation}
L_{A^*}N^j=0 \quad \textrm{i.e.}\quad [A^*,  N^j(X,Y)]=N^j([A^*,X],Y)+N^j(X,[A^*,Y]) \qquad \forall X,Y\in {\mathfrak{X}}(M).
\end{equation}
 The condition $[jX,N]-j[X,N] \in {\mathcal{I}m}N^j$ for all $X\in {\mathfrak{X}}(M)$ and for all $N\in {\mathcal{I}m}N^j$ will be satisfied at every point if it is satisfied at the base point.
 Since  $[jX,N]-j[X,N]$ is tensorial in $X$,  it is enough to verify it for $X$ in the space of fundamental vector fields; since it remains true if one multiplies $N\in {\mathcal{I}m}N^j$ by a function, the condition will be satisfied if and only if  
 \begin{equation}
 \left([jA^*,N^j(B^*,C^*)]-j[A^*,N^j(B^*,C^*)]\right)_{x_0} \in \Image N^j_{x_0} \qquad \forall A,B,C\in \g.
 \end{equation}
Introducing again  an auxiliary torsion-free linear connexion $\nabla$ on the manifold $M$, we have
\begin{eqnarray*}
\left([jA^*,N^j(B^*,C^*)]\right)_{x_0}&=&\left([(JA)^*,N^j(B^*,C^*)]    -\nabla_{N^j(B^*,C^*)}(jA^*-(JA)^*)\right)_{x_0}\\
&=&\left([(JA)^*,N^j(B^*,C^*)]+\nabla_{(N^J(B,C))^*}(jA^*-(JA)^*)\right)_{x_0}\\
&=&[(JA)^*,N^j(B^*,C^*)]_{x_0}+[N^J(B,C)^*,jA^*-(JA)^*]_{x_0}\\
&=&\left([(JA)^*,N^j(B^*,C^*)]+j[N^J(B,C)^*,A^*]-[N^J(B,C),JA]^*\right)_{x_0}\\
&=&\left(N^j([JA,B]^*,C^*)+N^j(B^*,[JA,C]^*)\right)_{x_0}\\
&&\qquad+\left(J[N^J(B,C),A]-[N^J(B,C),JA]\right)^*_{x_0}\\
&=&\left(-N^J([JA,B],C)-N^J(B,[JA,C])\right.\\
&&\qquad\left.+J[N^J(B,C),A]-[N^J(B,C),JA]\right)^*_{x_0}.
\end{eqnarray*}
On the other hand, we have  
\begin{eqnarray*}
\left(j[A^*,N^j(B^*,C^*)]\right)_{x_0}&=&\left(jN^j([A,B]^*,C^*)+jN^j(B^*,[A,C]^*)\right)_{x_0}\\
&=&\left(-JN^J([A,B],C)-JN^J(B,[A,C])\right)^*_{x_0}.
\end{eqnarray*}
Hence 
\begin{eqnarray*}
\left([jA^*,N^j(B^*,C^*)]-j[A^*,N^j(B^*,C^*)]\right)_{x_0}
 &=&\left(-N^J([JA,B],C)-N^J(B,[JA,C]) \right.\\
 &&\qquad  +J[N^J(B,C),A]-[N^J(B,C),JA]\\
 &&\qquad \left. +JN^J([A,B],C)+JN^J(B,[A,C])\right)^*_{x_0}\\
 \end{eqnarray*}
 and this is in $\Image N^j_{x_0}=(\Image N^J)^*_{x_0}$ iff 
 \begin{equation*}
 [JA,N^J(B,C)]-J[A, N^J(B,C)] \in \Image N^J+\h.
 \end{equation*}
\end{proof}

\end{document}